\documentclass{amsart}

\usepackage{hyperref}

\usepackage{amsmath}
\usepackage{amssymb}
\usepackage{amsthm}

\usepackage{tikz}
\usetikzlibrary{fadings}

\def\OO{{\mathcal O}}

\def\IN{{\Bbb{N}}}

\def\IR{{\Bbb{R}}}
\def\IS{{\Bbb{S}}}

\def\TO{\Longrightarrow}
\def\In{\subseteq}

\def\into{\hookrightarrow}

\def\prefix{\sqsubseteq}

\def\mto{\rightrightarrows}

\def\K{{\mathsf{{K}}}}

\def\T{{\mathsf{{T}}}}
\def\U{{\mathsf{{U}}}}

\def\w{{\mathsf{{w}}}}

\def\C{{\mathsf{{C}}}}

\def\LPO{\mathsf{LPO}}
\def\LLPO{\mathsf{LLPO}}

\def\C{\mathsf{C}}

\def\w{\mathsf{w}}

\def\leqm{\mathop{\leq_{\mathrm{m}}}}

\def\leqT{\mathop{\leq_{\mathrm{T}}}}

\def\equivT{\mathop{\equiv_{\mathrm{T}}}}

\def\leqW{\mathop{\leq_{\mathrm{W}}}}
\def\equivW{\mathop{\equiv_{\mathrm{W}}}}

\def\leqSW{\mathop{\leq_{\mathrm{sW}}}}
\def\equivSW{\mathop{\equiv_{\mathrm{sW}}}}

\def\lW{\mathop{<_{\mathrm{W}}}}

\def\leqm{\mathop{\leq_{\mathrm{m}}}}

\def\leqT{\mathop{\leq_{\mathrm{T}}}}

\def\equivT{\mathop{\equiv_{\mathrm{T}}}}

\def\leqW{\mathop{\leq_{\mathrm{W}}}}
\def\equivW{\mathop{\equiv_{\mathrm{W}}}}

\def\leqSW{\mathop{\leq_{\mathrm{sW}}}}
\def\equivSW{\mathop{\equiv_{\mathrm{sW}}}}

\def\lW{\mathop{<_{\mathrm{W}}}}

\def\id{{\mathrm{id}}}

\def\dom{{\mathrm{dom}}}
\def\range{{\mathrm{range}}}
\def\graph{{\mathrm{graph}}}

\newcommand{\PO}[1]{{{\mathbf\Pi}^0_{#1}}}

\def\NON{\mathsf{NON}}

\def\NRNG{\mathsf{NRNG}}
\def\DIS{\mathsf{DIS}}

\def\AC{\mathsf{AC}}

\def\AD{\mathsf{AD}}
\def\ZFC{\mathsf{ZFC}}
\def\ZF{\mathsf{ZF}}
\def\DC{\mathsf{DC}}
\def\BP{\mathsf{BP}}

\def\r{\mathrm{r}}

\makeatletter
\newcommand{\douwidehat}[2]{%
  \sbox0{$\m@th#1\widehat{\hphantom{#2}}$}%
  \sbox2{$\m@th#1x$}
  \sbox4{$\m@th#1#2$}
  \dimen0=\ht0
  \advance\dimen0 -.8\ht2
  \dimen2=\dp4
  \rlap{%
    \raisebox{\dimexpr\dimen0-\dimen2}{%
      \scalebox{1}[-1]{\box0}%
    }%
  }%
  {#2}%
}
\makeatother

\usepackage{catchfile}

\usepackage{lineno}

\title{The Discontinuity Problem}

\author{Vasco Brattka}
\address{Faculty of Computer Science, Universit\"at der Bundeswehr M\"unchen, Germany and
Department of Mathematics and Applied Mathematics, University of Cape Town, South Africa}
\email{Vasco.Brattka@cca-net.de}

\begin{document}

\theoremstyle{definition}
\newtheorem{theorem}{Theorem}
\newtheorem{definition}[theorem]{Definition}
\newtheorem{problem}[theorem]{Problem}
\newtheorem{assumption}[theorem]{Assumption}
\newtheorem{corollary}[theorem]{Corollary}
\newtheorem{proposition}[theorem]{Proposition}
\newtheorem{lemma}[theorem]{Lemma}
\newtheorem{observation}[theorem]{Observation}
\newtheorem{fact}[theorem]{Fact}
\newtheorem{question}[theorem]{Question}
\newtheorem{example}[theorem]{Example}
\newtheorem{convention}[theorem]{Convention}
\newtheorem{conjecture}[theorem]{Conjecture}


\begin{abstract}
Matthias Schr\"oder has asked the question whether there is a weakest discontinuous problem in the 
continuous version of the Weihrauch lattice. Such a problem can be considered as the weakest unsolvable problem.
We introduce the {\em discontinuity problem}, and we show that it 
is reducible exactly to the {\em effectively  discontinuous problems}, defined in a suitable way.
However, in which sense this answers Schr\"oder's question sensitively depends on the axiomatic framework that is chosen, and it is a positive answer
if we work in Zermelo-Fraenkel set theory with dependent choice and the axiom of determinacy $\AD$.
On the other hand, using the full axiom of choice, one can construct problems which are discontinuous, but not
effectively so. 
Hence, the exact situation at the bottom of the Weihrauch lattice sensitively depends on the axiomatic setting that we choose.
We prove our result using a variant of Wadge games for mathematical problems. While the existence of a winning strategy for
player II characterizes continuity of the problem (as already shown by Nobrega and Pauly), the existence of a winning strategy for player I characterizes effective discontinuity of the problem. 
By {\em Weihrauch determinacy} we understand the condition that every problem is either continuous or
effectively discontinuous. This notion of determinacy is a fairly strong notion, as it is not only implied by the axiom of determinacy $\AD$, but
it also implies Wadge determinacy. We close with a brief discussion of generalized notions of productivity.
\end{abstract}

\maketitle

\setcounter{tocdepth}{1}
\tableofcontents

\section{Introduction}

The Weihrauch lattice has been used as a computability theoretic framework to analyze the uniform computational
content of mathematical problems from many different areas of mathematics, and it can also be seen as a uniform variant
of reverse mathematics (a recent survey on Weihrauch complexity can be found in \cite{BGP21}). 
The notion of a mathematical problem has a very general definition in this approach.

\begin{definition}[Problems]
\label{def:problem}
A {\em problem} is a multi-valued function $f:\In X\mto Y$ on represented spaces
$X,Y$ that has a realizer.
\end{definition}

We recall that by a {\em realizer} $F:\In\IN^\IN\to\IN^\IN$ of $f$, we mean a function $F$ that satisfies
$\delta_YF(p)\in f\delta_X(p)$ for all $p\in\dom(f\delta_X)$, where $\delta_X:\In\IN^\IN\to X$ and ${\delta_Y:\In\IN^\IN\to Y}$
are the representations of $X$ and $Y$, respectively (i.e., partial surjective maps onto $X$ and $Y$, respectively).

We note that we have added the condition here that a problem has to have a realizer, since
we want to prove all our results over the base theory of 
Zermelo-Fraenkel set theory ($\ZF$) together with the axiom of dependent choice ($\DC$), if not otherwise mentioned.
These axioms suffice to prove most results in Weihrauch complexity. Typically, the full axiom of choice ($\AC$)
is freely used in Weihrauch complexity, often just to guarantee the existence of some realizer.
By $\langle p,q\rangle$ we denote the usual
pairing function on $\IN^\IN$, defined by $\langle p,q\rangle(2n)=p(n)$, $\langle p,q\rangle(2n+1)=q(n)$
for all $p,q\in\IN^\IN,n\in\IN$.
Weihrauch reducibility can now be defined as follows. 

\begin{definition}[Weihrauch reducibility]
Let $f:\In X\mto Y$ and $g:\In W\mto Z$ be problems. Then $f$ is called {\em Weihrauch reducible}
to $g$, in symbols $f\leqW g$, if there are computable $H,K:\In\IN^\IN\to\IN^\IN$ such that
$H\langle \id,GK\rangle$ is a realizer of $f$ whenever $G$ is a realizer of $g$.
Analogously, one says that $f$ is {\em strongly Weihrauch reducible} to $g$, in symbols $f\leqSW g$, if the expression $H\langle \id,GK\rangle$
can be replaced by $HGK$. Both versions of the reducibility have continuous counterparts,
where one requires $H,K$ only to be continuous and these reducibilities are denoted by $\leq_\mathrm{W}^*$ and $\leq_\mathrm{sW}^*$,
respectively.
\end{definition}

The continuous version of Weihrauch reducibility has always been studied alongside the computable
version, and all four reducibilities induce a lattice structure (see \cite{BGP21} for references). 
Normally, the {\em Weihrauch lattice} refers
to the lattice induced by $\leqW$, but here we will freely use this term also for the lattice structure induced by $\leq_\mathrm{W}^*$. If we want to be more precise, we will call the latter the {\em continuous Weihrauch lattice}.\footnote{This should not be misunderstood such that the structure is continuous as a lattice, but it just indicates that we
refer to the lattice structure induced by the continuous version of the reducibility.}
Even though this lattice has been studied for about 30 years, very little is known about the structure of the lattice closer
towards the bottom. Indeed Matthias Schr\"oder has asked the following question\footnote{The original question is phrased slightly differently, but we interpret it in the intended way.}~\cite[Question~5.9]{BDMP19}.

\begin{question}[Matthias Schr\"oder 2018]
Does there exist a discontinuous problem $f$ such that $f\leq_{\rm W}^*g$ holds for any other 
discontinuous problem $g$?
\end{question}

Here a problem is called {\em continuous} if it has a continuous realizer and {\em discontinuous} otherwise.
It is clear that the degree $0$ of the nowhere defined problems $f$ is the bottom degree of the (continuous)
Weihrauch lattice. The second lowest degree, sometimes called $1$, is the degree of the identity $\id:\IN^\IN\to\IN^\IN$
that includes all somewhere defined continuous problems.\footnote{For the computable version and the strong continuous version of Weihrauch reducibility the continuous problems with non-empty domain do not form a single equivalence class, however together with $0$ they still form the cone below $\id$.}
Essentially Schr\"oder's question is whether there is a third degree $2$ such that the continuous Weihrauch lattice starts 
with the linear ordered structure $0<1<2$. We will prove that under certain conditions this is indeed so, namely there is such a third degree given by the problem 
\[\DIS:\IN^\IN\mto\IN^\IN,p\mapsto\{q\in\IN^\IN:\U(p)\not=q\}\]
that we call the {\em discontinuity problem} (see Definition~\ref{def:dis}). 
Here $\U:\In\IN^\IN\to\IN^\IN$ is a fixed universal computable function.

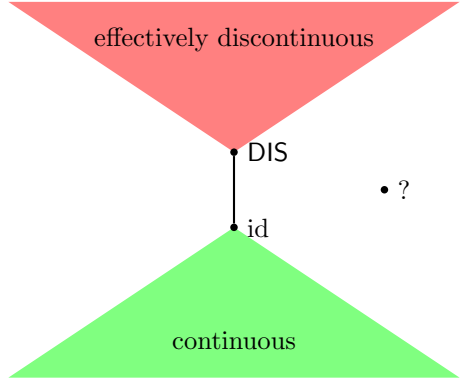
\begin{figure}
\begin{tikzpicture}
\fill[red!50,path fading=north] (0,0) -- (-3,2) -- (3,2) -- cycle;
\fill[green!50,path fading=south] (0,-1) -- (-3,-3) -- (3,-3) -- cycle;
\node[fill,circle,scale=0.3,label=right:$\DIS$] (DIS) at (0,0) {};
\node[fill,circle,scale=0.3,label=right:$\id$] (id) at (0,-1) {};
\node (1) at (0,1.5) {effectively discontinuous};
\node (2) at (0,-2.5) {continuous};
\node[fill,circle,scale=0.3,label=right:$?$]  (3) at (2,-0.5) {};
\draw [-,thick] (DIS) -- (id);
\end{tikzpicture}
\caption{Problems $f:\In X\mto Y$ with respect to $\leq_{\rm W}^*$ and $\leq_{\rm sW}^*$.} 
\label{fig:DIS}
\end{figure}

In section~\ref{sec:effective-discontinuity} we prove that the discontinuity problem $\DIS$ 
characterizes effectively discontinuous problem in the following sense (see Theorem~\ref{thm:effective-discontinuity}).

\begin{theorem}[Continuity and effective discontinuity]
\label{thm:dichotomy}
Let $f:\In X\mto Y$ be a problem. Then we obtain
\begin{enumerate}
\item $f\leq_\mathrm{W}^*\id\iff f$ is continuous,
\item $\DIS\leq_\mathrm{W}^*f\iff f$ is effectively discontinuous.
\end{enumerate}
\end{theorem}

The diagram in Figure~\ref{fig:DIS} illustrates the situation.
Here {\em effective discontinuity} of a problem $f$ is defined in a very natural way (see Definition~\ref{def:computable-discontinuity}) using a continuous {\em discontinuity function} $D:\IN^\IN\to\IN^\IN$ that has to produce an input to every given potential continuous realizer of $f$ on which this realizer fails. 

In some sense the notion of effective discontinuity is reminiscent of the notion of {\em productivity}
from classical computability theory, which can be regarded as the property of being ``effectively not c.e.''.
Indeed a well-known theorem of Myhill~\cite{Myh55} gives us the analog of Theorem~\ref{thm:dichotomy} for 
subsets $A\In\IN$ and many-one reducibility $\leqm$ \cite[Theorem~2.4.6]{Soa16}, \cite[Theorem~2.6.6]{Wei87}.
Here the {\em halting problem} $\K\In\IN$ plays the counterpart of the identity and $\IN\setminus\K$ the counterpart of $\DIS$.

\begin{theorem}[Myhill 1955]
\label{thm:Myhill}
Let $A\In\IN$. Then we obtain
\begin{enumerate}
\item $A\leqm \K\iff A$ is c.e.,
\item $\IN\setminus \K\leqm A\iff A$ is productive.
\end{enumerate}
\end{theorem}

The diagram in Figure~\ref{fig:productive} illustrates the situation.
In fact, Theorem~\ref{thm:dichotomy} is proved with the help of the recursion theorem, 
in a similar way as Theorem~\ref{thm:Myhill}. However, somewhat surprisingly, the recursion theorem is used
for the implication ``$\Longleftarrow$'' in the proof of Theorem~\ref{thm:Myhill}~(2) and for the direction ``$\TO$'' in the proof of Theorem~\ref{thm:dichotomy}~(2).

\begin{figure}
\begin{tikzpicture}
\fill[red!50,path fading=north] (0,0) -- (-3,2) -- (3,2) -- cycle;
\fill[green!50,path fading=south] (-2,0) -- (-5,-3) -- (1,-3) -- cycle;
\node[fill,circle,scale=0.3,label=right:$\IN\setminus\K$] (NK) at (0,0) {};
\node[fill,circle,scale=0.3,label=right:$\K$] (K) at (-2,0) {};
\node (1) at (0,1.5) {productive};
\node (2) at (-2,-2.5) {computably enumerable};
\node[fill,circle,scale=0.3,label=right:immune set]  (3) at (2,-1.5) {};
\end{tikzpicture}
\caption{Sets $A\In\IN$ with respect to many-one reducibility $\leq_{\rm m}$.} 
\label{fig:productive}
\end{figure}
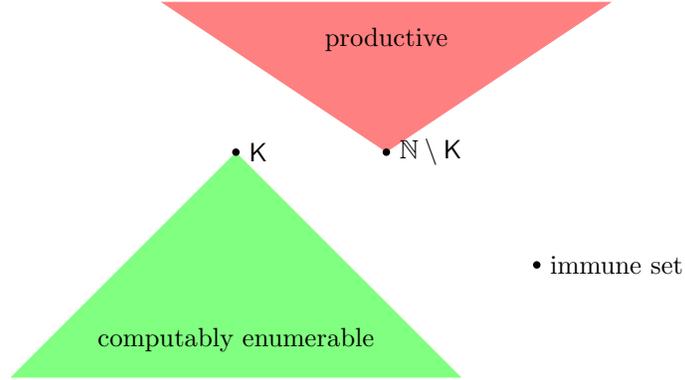

It is well known that Theorem~\ref{thm:Myhill} does not express a dichotomy, i.e., there are sets which are
neither c.e.\ nor productive.
A set $A\In\IN$ is called {\em immune} if it is infinite, but does not include an infinite c.e.\ set \cite{Soa16,Wei87}.
By a classical construction of Post immune sets exist \cite{Pos44}, \cite[Theorem~5.2.3]{Soa16}, \cite[Theorem~2.3.7]{Wei87} and they are clearly neither c.e.\ nor productive, since any productive set contains an infinite c.e.\ subset (see Figure~\ref{fig:productive}).

Now a key question for us is whether there is a counterpart of immune sets in our situation, i.e., whether there are problems $f$
that are discontinuous, but not effectively so (see the question mark in Figure~\ref{fig:DIS}).
In fact, Post's construction of a simple set has some similarity to the construction of so-called {\em Bernstein sets} that can actually be used to construct discontinuous problems that are not effectively discontinuous (see Corollary~\ref{cor:Bernstein}).
This leads to the following counterexample (see Theorem~\ref{thm:AC}).

\begin{theorem} 
Assuming $\ZFC$ there are problems $f:\IN^\IN\mto\IN^\IN$ that are continuous, but not effectively so.
\end{theorem}

In section~\ref{sec:games} we prove that the axiom of choice is actually required for such a construction.
This can be achieved with the help of a variant of Wadge games for problems $f:\In X\mto Y$, originally
considered by Nobrega and Pauly~\cite{Nob18,NP19}. In fact, we can prove the following result (see Theorem~\ref{thm:Wadge-game}),
part (1) of which is already due to Nobrega and Pauly.

\begin{theorem}[Wadge games]
We consider the Wadge game of a given problem ${f:\In X\mto Y}$.
Then the following hold:
\begin{enumerate}
\item $f$ is continuous $\iff$ Player II has a winning strategy for $f$,
\item $f$ is effectively discontinuous $\iff$ Player I has a winning strategy for $f$.
\end{enumerate}
\end{theorem}

This result implies that under the axiom of determinacy ($\AD$), which states that every Gale-Stewart game is determined, i.e., either player I or player II has a winning strategy, we really obtain a dichotomy between continuity and effective discontinuity
(see Corollary~\ref{cor:determinacy}).

\begin{corollary}
\label{cor:ZF-DC-AD}
In $\ZF+\DC+\AD$ every problem $f:\In X\mto Y$ is either continuous or effectively discontinuous.
\end{corollary}

This can be proved by a suitable reduction of the Wadge game to a Gale-Stewart game.
By {\em Weihrauch determinacy} we understand the condition that every problem $f:\In X\mto Y$ is either continuous or effectively discontinuous. Then Corollary~\ref{cor:ZF-DC-AD} can also be rephrased such that $\AD$ implies Weihrauch determinacy. On the other hand, it is easy to see that Wadge games for problems $f:\In X\mto Y$ generalize Wadge
games for subsets $A,B\In\IN^\IN$ as originally considered by Wadge and hence Weihrauch determinacy implies 
Wadge determinacy, which means that every Wadge game for subsets $A,B\In\IN^\IN$ is determined (see Figure~\ref{fig:determinacy}).

\begin{figure}[thb]
\begin{tikzpicture}[scale=1]
\node (AD) at (0,2) {Axiom of Determinacy $\AD$};
\node (WeiD) at (0,1) {Weihrauch determinacy};
\node (WD) at (0,0) {Wadge determinacy};
\node (PSP) at (2.5,-1) {Perfect Subset Property};
\node (CC) at (-2.5,-1) {Axiom of Countable Choice};
\draw[thick,->] (AD) -- (WeiD);
\draw[thick,->] (WeiD) -- (WD);
\draw[thick,->] (WD) -- (PSP);
\draw[thick,->] (WD) -- (CC);
\end{tikzpicture}
\caption{Determinacy properties in $\ZF+\DC$.} 
\label{fig:determinacy}
\end{figure}
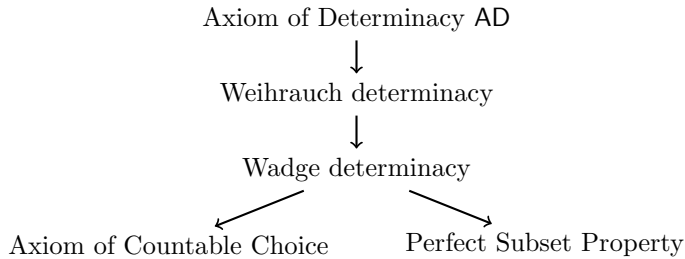

We leave it open how Weihrauch determinacy is exactly related to the other mentioned notions of determinacy. 
We just mention that it is not known whether Wadge determinacy implies $\AD$, and in fact there is 
the following conjecture attributed to Solovay (by Andretta~\cite{And03}\footnote{Andretta also proved that Wadge determinacy implies the axiom of countable choice for Baire space~\cite[Theorem~3]{And03}, whereas Wadge proved that 
Wadge determinacy implies the perfect subset property for Baire space~\cite[Theorem~II.C.2]{Wad83}.}).

\begin{conjecture}[Solovay]
In $\ZF+V=L(\IR)$ Wadge determinacy is equivalent to the axiom of determinacy $\AD$.
\end{conjecture}

We have some partial results (see Proposition~\ref{prop:Lipschitz-AD}) that suggest that Weihrauch determinacy
is actually even closer related to $\AD$ than Wadge determinacy.

In any case, our results show that the exact situation at the bottom of the continuous Weihrauch lattice 
sensitively depends on the underlying axioms. Using $\AD$ we obtain the following result (as a consequence of Corollary~\ref{cor:ZF-DC-AD}).

\begin{theorem}
\label{thm:bottom}
In $\ZF+\DC+\AD$ the continuous Weihrauch lattice starts with three linearly ordered degrees, represented by
$0\leq_\mathrm{W}^*\id\leq_\mathrm{W}^*\DIS$.
\end{theorem}

If we replace $\AD$ by the axiom of choice $\AC$, then the linear part of the order at the bottom is just $0\leq_\mathrm{W}^*\id$
and the situation becomes more complicated afterwards (and  by Theorem~\ref{thm:AC} this is even true if we move to the coarser parallelized version of the Weihrauch lattice).

We briefly summarize the structure of this article. In the following section~\ref{sec:recursion} we 
provide the version of the recursion theorem that we are going to use for the proof of our characterization 
of effectively discontinuous problems via the discontinuity problem.
We also introduce the universal function $\U$ and other related concepts.
In section~\ref{sec:effective-discontinuity} we introduce the discontinuity problem, the notion of effective discontinuity,
and we prove related results. In section~\ref{sec:games} we characterize effective discontinuity using Wadge games
and we study the relation to determinacy of other games such as Lipschitz games and Gale-Stewart games.
In section~\ref{sec:characteristic} we briefly discuss computable discontinuity of characteristic functions and we indicate
how this is related to (suitable generalizations of) the notion of productivity.
In the conclusions~\ref{sec:conclusions} we mention a number of open open problems and suggestions
for further directions of research.

\section{The Universal Function and the Recursion Theorem}
\label{sec:recursion}

We recall that a function $F:\In\IN^\IN\to\IN^\IN$ is {\em computable}, if there is some computable monotone word function $f:\IN^*\to\IN^*$
that {\em approximates} $F$ in the sense that $F(p)=\sup_{w\prefix p}f(w)$ holds for all $p\in\dom(F)$. 
Likewise, $F$ is continuous if and only if an analogous condition holds for an arbitrary monotone word function $f$.
Hence, we can define a representation $\Phi$ of the set of all continuous functions $F:\In\IN^\IN\to\IN^\IN$ (with natural domains\footnote{We note that for mere cardinality reasons
there is no representation of {\em all} partial continuous $F:\In\IN^\IN\to\IN^\IN$, but our representation $\Phi$ represents sufficiently many such functions in the sense that it 
contains an extension of any continuous partial function on Baire space.}) 
by encoding graphs of monotone word functions $f$ into names of $F$. In other words, if $\w:\IN\to\IN^*$ is a standard bijective numbering 
of $\IN^*$, then $p=\langle n_0,k_0\rangle\langle n_1,k_1\rangle...$ is a name of an extension $\Phi_p$ of $F$ if $F$ is approximated by some monotone $f:\IN^*\to\IN^*$ with
$f(\w_{n_i})=\w_{k_i}$ and for each $p\in\dom(F)$ and $n\in\IN$ there is some $i\in\IN$ with $\w_{n_i}\prefix p$ and $|\w_{n_i}|>n$.
Here $\langle n,k\rangle:=\frac{1}{2}(n+k)(n+k+1)+k$ denotes the usual {\em Cantor pairing function} for $n,k\in\IN$. 
Intuitively, $F=\Phi_p$ means that $p$ is a listing of a sufficiently large portion of the graph of a monotone function $f:\IN^*\to\IN^*$ that approximates $F$.
In order to guarantee that $\Phi$ is a total representation, one still needs to clarify how to deal with inconsistent names $p$,
i.e., names for which there is no suitable word function $f$. 
Inconsistency can be recognized (i.e., inconsistent names form an open set) and hence
one can just consider those $p$ as names of the nowhere defined function.
See \cite[Definition~3.2.9]{Wei87} for the technical details of such a construction of $\Phi$.

Now we can define a computable {\em universal function} 
\[\U:\In\IN^\IN\to\IN^\IN,\langle q,p\rangle\mapsto\Phi_q(p)\] 
for all $p,q\in\IN^\IN$ \cite[Theorem~3.2.16~(1)]{Wei87}. Here $\langle q,p\rangle:=q(0)p(0)q(1)p(1)...$ denotes the standard pairing function on Baire space.
Weihrauch~\cite[Theorems~3.5, 2.10, Corollary~2.11]{Wei85} (see also \cite[Theorem~3.2.16]{Wei87}) proved the following version of the smn-theorem for the representation $\Phi$
that comes in a version for computable and a version for continuous functions.

\begin{theorem}[smn-Theorem]
\label{thm:smn}
For every computable (continuous) partial function $F:\In\IN^\IN\to\IN^\IN$ there exists a computable (continuous) total function
$S:\IN^\IN\to\IN^\IN$ such that $\Phi_{S(q)}(p)=F\langle q,p\rangle$ for all $p,q\in\IN^\IN$.
\end{theorem}

Among other things this result implies that $\Phi$ is precomplete.
We recall that in general a representation $\delta:\In\IN^\IN\to X$ of a set $X$ is called {\em precomplete}, if for every computable $F:\In\IN^\IN\to\IN^\IN$ there exists a total computable $G:\IN^\IN\to\IN^\IN$ such that $\delta F(p)=\delta G(p)$ for all $p\in\dom(F)$. In other words, precomplete representations are exactly those under which partial computable functions can be extended to total ones. 

Using the smn-theorem one can prove the following uniform version of the recursion theorem along the same lines as the classical recursion theorem.
It is an immediate corollary of a more general result due to Kreitz and Weihrauch~\cite[Theorem~3.4]{KW85}  (see also \cite[Theorem~3.3.20]{Wei87}), which characterizes precomplete
representations following Ershov's characterization of precomplete numberings.

\begin{theorem}[Uniform recursion theorem]
\label{thm:recursion-theorem}
There exists a total computable function $T:\IN^\IN\to\IN^\IN$ such that 
$\Phi_{T(p)}=\Phi_{\Phi_pT(p)}$
for all $p\in\IN^\IN$ such that $\Phi_p$ is total.
\end{theorem}

As a corollary of this theorem we obtain the following parameterized version of the recursion theorem that also comes in a version for computable and a version for continuous functions.

\begin{corollary}[Parameterized recursion theorem]
\label{cor:recursion-theorem}
For every computable (continuous) function $F:\In\IN^\IN\to\IN^\IN$ there exists a total computable (continuous) function $R:\IN^\IN\to\IN^\IN$ such that
$\U R(q)=F\langle q,R(q)\rangle$ for all $q\in\IN^\IN$.
\end{corollary}
\begin{proof}
We prove the version of the statement for continuous functions. For computable functions one just has to replace the word ``continuous'' by ``computable'' in all occurrences.
Let $F:\In\IN^\IN\to\IN^\IN$ be continuous. Then by a double application of the smn-theorem (Theorem~\ref{thm:smn}) there is a total continuous $S:\IN^\IN\to\IN^\IN$ such that
$\Phi_{\Phi_{S(q)}(r)}(p)=F\langle p,\langle r,q\rangle\rangle$ for all $p,q\in\IN^\IN$. Let $T:\IN^\IN\to\IN^\IN$ be the computable function from the
recursion theorem (Theorem~\ref{thm:recursion-theorem}). Then $R:\IN^\IN\to\IN^\IN$ with $R(q):=\langle TS(q),q\rangle$ for all $q\in\IN^\IN$ is continuous
and satisfies
\[\U R(q)=\Phi_{TS(q)}(q)=\Phi_{\Phi_{S(q)}TS(q)}(q)=F\langle q,\langle TS(q),q\rangle\rangle=F\langle q,R(q)\rangle\] 
for all $q\in\IN^\IN$.
\end{proof}

In the next section we will use this parameterized version of the recursion theorem to prove our characterization of
effectively discontinuous functions.

\section{Effectively Discontinuous Problems}
\label{sec:effective-discontinuity}

We now introduce a concept of {\em computable} (and {\em effective}) {\em discontinuity}. 
These are strengthenings of the concept of discontinuity in the sense that the discontinuity is witnessed by
a continuous function $D:\IN^\IN\to\IN^\IN$. 
A computably discontinuous problem $f$ is supposed to have no continuous realizer $\Phi_q$ and the computable discontinuity function $D:\IN^\IN\to\IN^\IN$
computes for every candidate $\Phi_q$ a witnessing input $D(q)$ that shows that the candidate $\Phi_q$ fails to realize $f$ on that particular input.

\begin{definition}[Computable discontinuity]
\label{def:computable-discontinuity}
Let $(X,\delta_X)$ and $(Y,\delta_Y)$ be represented spaces.
A problem $f:\In X\mto Y$ is called {\em computably discontinuous} ({\em effectively discontinuous}) if there is a computable (continuous) $D:\IN^\IN\to\IN^\IN$ such that for all $q\in\IN^\IN$ we obtain
\[D(q)\in\dom(f\delta_X)\mbox{ and }\delta_Y\Phi_qD(q)\not\in f\delta_XD(q).\]
In this case the function $D$ is called a {\em discontinuity function} of $f$.
\end{definition}

We emphasize that we consider the condition $\delta_Y\Phi_qD(q)\not\in f\delta_XD(q)$ also as satisfied when the left-hand side is undefined. 
If one does not apply this convention, then one has to write out the condition somewhat more detailed as 
\[D(q)\in\dom(f\delta_X)\mbox{ and }(D(q)\in\dom(\delta_Y\Phi_q)\TO\delta_Y\Phi_qD(q)\not\in f\delta_XD(q)).\]
Clearly, every computably discontinuous problem $f$ is effectively discontinuous, and every
effectively discontinuous problem is discontinuous.
The definition of computable (effective) discontinuity is such that a problem $f:\In X\mto Y$ has the respective
property if and only if its {\em realizer version} $f^{\mathrm r}:\In\IN^\IN\mto\IN^\IN$ with $f^{\mathrm r}:=\delta_Y^{-1}\circ f\circ\delta_X$ 
has the property. This implies that it suffices to study the concepts of computable and effective discontinuity on Baire space $\IN^\IN$.
We will prove that effective and computable discontinuity can be both characterized in terms of the following
discontinuity problem.

\begin{definition}[Discontinuity problem]
\label{def:dis}
We define the {\em discontinuity problem} by
$\DIS:\IN^\IN\mto\IN^\IN,p\mapsto\{q\in\IN^\IN:\U(p)\not=q\}$.
\end{definition}

We note that $\DIS$ is total, i.e., for instances $p\not\in\dom(\U)$ the problem $\DIS$ can provide arbitrary $q\in\IN^\IN$
as solutions.
It is a direct consequence of the parameterized recursion theorem from Corollary~\ref{cor:recursion-theorem}
that $\DIS$ is computably discontinuous.

\begin{proposition} 
\label{prop:DIS-computably-discontinuous} 
$\DIS$ is computably discontinuous.
\end{proposition}
\begin{proof}
By Corollary~\ref{cor:recursion-theorem} there is a computable function $D:\IN^\IN\to\IN^\IN$
such that $\U D(p)=\U\langle p,D(p)\rangle=\Phi_pD(p)$ for all $p\in\IN^\IN$. This function $D$ is hence a computable discontinuity function for $\DIS$.
\end{proof}

We generalize this observation with the following result that is illustrated in the diagram in Figure~\ref{fig:DIS}.

\begin{theorem}[Effective discontinuity]
\label{thm:effective-discontinuity}
Let $f:\In X\mto Y$ be a problem. Then:
\begin{enumerate}
\item $\DIS\leqW f\iff f$ is computably discontinuous,
\item $\DIS\leq_{\rm W}^* f\iff f$ is effectively discontinuous.
\end{enumerate}
In both cases one can replace ${\rm W}$ by its strong counterpart ${\rm sW}$.
\end{theorem}
\begin{proof}
Since $f\equivSW f^{\mathrm r}$ and $f$ is computably (effectively) discontinuous if and only if $f^\mathrm{r}$ is, 
it suffices to prove both statements for problems of type $f:\In\IN^\IN\mto\IN^\IN$.\medskip\\
(1) 
``$\TO$'' Let $\DIS\leqW f$ hold via computable $H,K:\In\IN^\IN\to\IN^\IN$, i.e.,
$H\langle p,FK(p)\rangle\not=\U(p)$ and $H\langle p,FK(p)\rangle$ defined for all $p\in\IN^\IN$ and every realizer $F$  of $f$.
In particular, $K$ is total and $K(p)\in\dom(f)$ for every $p\in\IN^\IN$.
By the parameterized recursion theorem from Corollary~\ref{cor:recursion-theorem} there is some computable $R:\IN^\IN\to\IN^\IN$ such that
\[H\langle R(q),\U\langle q,KR(q)\rangle\rangle=\U R(q)\] 
for all $q\in\IN^\IN$. Then $KR$ is a total computable function and $KR(q)\in\dom(f)$ for all $q\in\IN^\IN$.
Let $q\in\IN^\IN$ be such that $KR(q)\in\dom(\Phi_q)$ and let us assume that $\U\langle q,KR(q)\rangle=\Phi_qKR(q)\in fKR(q)$.
Since $f$ is realizable, there is a realizer $F$ of $f$ with $FKR(q)=\U\langle q,KR(q)\rangle$ and hence 
\[H\langle R(q),\U\langle q,KR(q)\rangle\rangle=H\langle R(q),FKR(q)\rangle\not=\U R(q)\] 
follows
by the choice of $H,K$ in contradiction to the choice of $R$.
Hence the assumption was wrong, i.e., $\Phi_qKR(q)\not\in fKR(q)$ and $D:=KR$ is a total computable discontinuity function for $f$.\\
``$\Longleftarrow$''
Let $f$ be computably discontinuous with a corresponding computable discontinuity function 
$D:\IN^\IN\to\IN^\IN$.
By the smn-theorem (Theorem~\ref{thm:smn}) there is a total computable $R:\IN^\IN\to\IN^\IN$ such that
$\U\langle R(p),q\rangle=\U(p)$ for all $p,q\in\IN^\IN$. 
Now for every realizer $F$ of $f$ and $p\in\IN^\IN$ we obtain $DR(p)\in\dom(f)$ and $FDR(p)\in fDR(p)$ and hence
\[\U(p)=\U\langle R(p),DR(p)\rangle=\Phi_{R(p)}DR(p)\not=FDR(p),\]
since $D$ is a discontinuity function.
Hence, $DR$ is a computable function that witnesses $\DIS\leqSW f$.\medskip\\
(2) The proof is literally the same as above, except that $H,K,R$ and $D$ are supposed to be continuous instead of computable.
\end{proof}

We note that the proof of ``$\Longleftarrow$'' shows that we could replace the strong versions of Weihrauch reducibility by an even stronger form  of reducibility that only uses the inner reduction function $K$ and no outer reduction function $H$.

Now we can ask the question whether every discontinuous problem is automatically computably discontinuous, i.e.,
whether the discontinuity problem $\DIS$ is the smallest discontinuous problem with respect to the computable version of Weihrauch reducibility.
This is clearly not the case, as the following example shows. In fact, we can infer this from the existence of immune sets.
We recall that a set $A\In\IN$ is called {\em immune} if it is infinite but does not contain any infinite c.e.\ subset.

\begin{example}
\label{ex:immune}
Let $A\In\IN$ and let $f_A:\In[0,1]\to\IR$ be defined by $\dom(f_A):=\{2^{-n}:n\in A\}\cup\{0\}$ and
\[f_A(x):=\left\{\begin{array}{ll}
1 & \mbox{if $x=0$}\\
0 & \mbox{otherwise}
\end{array}\right..\]
Then we obtain:
\begin{enumerate}
\item $f_A$ discontinuous $\iff$ $A$ infinite.
\item $f_A$ computably discontinuous $\iff$ $A$ contains an infinite c.e.\ subset.
\item $f_A$ discontinuous and not computably discontinuous $\iff$ $A$ is immune.
\end{enumerate}
\end{example}
\begin{proof}
(1) is obvious and (3) follows from (1) and (2). For the proof of (2), we note that given a computable discontinuity function $D:\IN^\IN\to\IN^\IN$ for $f_A$, we get an infinite c.e.\ subset of $A$ as follows.
Given $n\in\IN$ and $A_n=\{0,1,2,...,n\}$ we can compute a name $q$ of a realizer $\Phi_q$ of the continuous
function $f_{A_n}$, and $D(q)$ has to provide a name of an input $2^{-k}\in\IN$ to $f_{A}$ on which the realizer $\Phi_q$ fails. 
But this means that $k\in A\setminus A_n$.  The collection of all those $k$ for $n=0,1,2,...$ forms an infinite c.e.\ subset of $A$. On the other hand, given an infinite c.e.\ subset $B\In A$ and a potential realizer $\Phi_q$ of $f_A$,
we can evaluate $\Phi_q$ successively on names of the compact interval $[0,2^{-k}]$ with $k=0,1,2,...$. Then either $\Phi_q$ eventually produces an output that excludes $0$ or that excludes $1$ or neither of this ever happens. In the first case, we can find some $m\in B$ with
$m>k$ and produce a name of $2^{-m}$ as output of $D(q)$, since this is a value on which $\Phi_q$ fails to realize $f_A(2^{-m})=0$.
In the second case, we produce a name of $0$ as output of $D(q)$, since this is then a value on which $\Phi_q$ fails to realize $f_A(0)=1$. As long as neither the first nor the second case occurs, we produce the open interval $(-2^{-k},2^{-k+1})$ as approximation of an output of $D(q)$. If the first and the second case never occurs, then this produces a name of $0$ as output of $D(q)$, which is then not in the domain of $\Phi_q$. Altogether, $D$ is a computable discontinuity function for $f_A$.
\end{proof}

The following result shows that with respect to the computable version of Weih\-rauch reducibility 
we can even get an infinite descending chain 
of discontinuous problems that are not computably discontinuous.
If we restrict the discontinuity problem to the Turing cone $[p]:=\{q\in\IN^\IN:p\leqT q\}$, 
we obtain an effectively discontinuous problem $\DIS|_{[p]}$. With increasing complexity of $p$ these problems get weaker.
By $p^{(n)}$ we denote the $n$--th {\em Turing jump} of $p\in\IN^\IN$.

\begin{proposition}
\label{prop:decreasing-DIS-chain}
$\DIS|_{[p^{(n+1)}]}\lW\DIS|_{[p^{(n)}]}$ and $\DIS|_{[p^{(n)}]}$ is effectively discontinuous for all $p\in\IN^\IN$ and $n\in\IN$, but not computably so for $n\geq1$.
\end{proposition}
\begin{proof}
It is clear that $\DIS|_{[p^{(n+1)}]}\leqSW\DIS|_{[p^{(n)}]}$ holds for all $n\in\IN$, as the 
former problem is a restriction of the latter one. We have $\DIS|_{[p^{(n)}]}\not\leq_\mathrm{W}\DIS|_{[p^{(n+1)}]}$,
as an instance $q\equivT p^{(n)}$ of $\DIS|_{[p^{(n)}]}$ cannot be computably mapped to an
instance of $\DIS|_{[p^{(n+1)}]}$.

On the other hand, every function $\Phi_q$ has names of Turing degree above any $p\in\IN^\IN$, and we can even continuously
determine such names by adding redundant information in the code $q$ that encodes $p$ (for instance by repeating the $n$--th entry of $q$ exactly $p(n)$ times). 
In other words, for every $p\in\IN^\IN$ there is a continuous function $R_p:\IN^\IN\to\IN^\IN$ such that $\Phi_q=\Phi_{R_p(q)}$ and $p\leqT R_p(q)$.
Hence $\DIS\leq_\mathrm{W}^*\DIS|_{[p]}$ holds for every $p\in\IN^\IN$. That is, $\DIS|_{[p]}$ is effectively discontinuous by Theorem~\ref{thm:effective-discontinuity}.
\end{proof}

Hence, we even have an infinite descending chain of effectively discontinuous problems below $\DIS$
with respect to the computable version of Weihrauch reducibility.
In particular, an effectively discontinuous problem does not need to be computably discontinuous.
The next question is whether there are discontinuous problems, which are not even effectively discontinuous.
The following proposition provides a sufficient condition for such an example.

\begin{proposition}
\label{prop:embedding-effectively-discontinuous}
Let $\iota:\IN^\IN\to\IN^\IN$ be injective and $B:=\range(\iota)$. We consider the problem
$f:\In\IN^\IN\mto\IN^\IN$ with the domain $\dom(f):=B$ and
\[f(p):=\left\{\begin{array}{ll}
\IN^\IN\setminus\{\Phi_{\iota^{-1}(p)}(p)\} & \mbox{if $p\in\dom(\Phi_{\iota^{-1}(p)})$}\\
\IN^\IN & \mbox{otherwise}
\end{array}\right.\]
for all $p\in B$. Then $f$ is discontinuous. Moreover, if $f$ is effectively discontinuous, then there is a continuous
embedding $g:2^\IN\into B$.
\end{proposition}
\begin{proof}
A function $D:\IN^\IN\to\IN^\IN$ is a discontinuity function for $f$ if and only if $\range(D)\In B$ and
\[D(q)\in\dom(\Phi_q)\TO\Phi_qD(q)=\Phi_{\iota^{-1}D(q)}D(q)\]
holds for all $q\in\IN^\IN$. Clearly, $D=\iota$ is a discontinuity function
for $f$, albeit not necessarily a continuous one. Nevertheless, this shows that $f$ is discontinuous.

Let us now assume that $D:\IN^\IN\to\IN^\IN$ is a continuous function with the above property. 
By the smn-theorem (Theorem~\ref{thm:smn}) there exists a computable total function $R:\IN^\IN\to\IN^\IN$ such that
$\Phi_{R(q)}(p)=q$ for all $p,q\in\IN^\IN$. Then we obtain 
\[q=\Phi_{R(q)}DR(q)=\Phi_{\iota^{-1}DR(q)}DR(q)\]
for all $q\in\IN^\IN$, which is only possible if $DR:\IN^\IN\to\IN^\IN$ is injective.
In particular, $DR|_{2^\IN}:2^\IN\to B$ is a continuous embedding of Cantor space into $B$.
\end{proof}

Hence, the existence of an injective map $\iota:\IN^\IN\to\IN^\IN$ with a range $B=\range(\iota)$
into which Cantor space cannot be continuously embedded, is sufficient to guarantee the existence
of a problem $f$ that is discontinuous, but not effectively discontinuous.

For this purpose it is sufficient to show that there exists a set $B\In\IN^\IN$, which violates the perfect subset property.
A set $B$ satisfies the {\em perfect subset property} if it is either countable or it contains a non-empty {\em perfect} subset, which
is a subset that is closed and has no isolated points. We recall that every $A\In\IN^\IN$ into which Cantor space can be continuously embedded contains a perfect subset.
It is a well-known fact that there are so-called {\em Bernstein sets}~\cite{Ber08a}, which violate the perfect subset property,
at least if we assume the axiom of choice~\cite[Exercise~8.24]{Kec95}.

\begin{fact}[Bernstein set]
\label{fact:Bernstein}
Assuming $\AC$ there exists a {\em Bernstein set} $B\In\IN^\IN$, which is a set $B$ such that $B$ as well as its complement $\IN^\IN\setminus B$ have non-empty intersection 
with every uncountable closed set $A\In\IN^\IN$. 
\end{fact}

The construction of a Bernstein set $B\In\IR$ provided in \cite[Exercise~8.24]{Kec95} works equally well for $B\In\IN^\IN$
and is by transfinite recursion. 
This construction necessarily requires the axiom of choice.
Indeed, if we assume $\BP$, i.e., that every subset $B\In\IN^\IN$ has the Baire property,
then no Bernstein set can exist \cite{Kec95}.
From the point of view of computability theory Bernstein sets play a similar r\^{o}le as immune sets.
Actually, together with Proposition~\ref{prop:embedding-effectively-discontinuous} and Fact~\ref{fact:Bernstein}
we directly obtain the following conclusion.

\begin{corollary}
\label{cor:Bernstein}
If we assume $\AC$, then there exists a problem $f:\In\IN^\IN\mto\IN^\IN$
that is discontinuous, but not effectively discontinuous.
\end{corollary}

Similarly as in the case of Proposition~\ref{prop:decreasing-DIS-chain}, one could object that the problem $f$ constructed here
is not genuinely less discontinuous than $\DIS$, but only ``simpler'' as instances are artificially made harder.
In other words, the problem $f$ considered here has a very complicated domain, which in the case of our proof of Corollary~\ref{cor:Bernstein}
is a Bernstein set.

With the next result we dispel this objection by constructing a total problem $f:\IN^\IN\mto\IN^\IN$
that is discontinuous, but not effectively so. Again, the construction is based on the axiom of choice,
and we directly perform a transfinite recursion.
We can arrange this construction even such that $f$ is parallelizable.
We recall that $f:\In\IN^\IN\mto\IN^\IN$ is {\em parallelizable} if and only if  $f\equivW\langle f\rangle$,
where $\langle f\rangle$ is defined by $\langle f\rangle:\In\IN^\IN\mto\IN^\IN,\langle p_0,p_1,...\rangle\mapsto\langle f(p_0),f(p_1),...\rangle$. 
By $|X|$ we denote the {\em cardinality} of a set $X$.

\begin{theorem}
\label{thm:AC}
Assuming $\AC$, there exists a total parallelizable $f:\IN^\IN\mto\IN^\IN$ that is discontinuous, but not 
effectively discontinuous.
\end{theorem}
\begin{proof}
It suffices to construct a total $f:\IN^\IN\mto\IN^\IN$ that is discontinuous and such that
$\langle f\rangle:\IN^\IN\mto\IN^\IN,\langle p_0,p_1,...\rangle\mapsto\langle f(p_0),f(p_1),...\rangle$ 
is not effectively discontinuous. Then $\langle f\rangle$ has the desired properties, since
it is discontinuous and parallelizable.   
We note that the set $\C(\IN^\IN,\IN^\IN)$ of total continuous functions $g:\IN^\IN\to\IN^\IN$ has continuum cardinality.
Hence, by the axiom of choice there is a transfinite enumeration $(r_\xi)_{\xi<2^{\aleph_0}}$ of $r_\xi\in\IN^\IN$
such that $\C(\IN^\IN,\IN^\IN)=\{\Phi_{r_\xi}:\xi<2^{\aleph_0}\}$.
A problem $f:\IN^\IN\mto\IN^\IN$ is discontinuous and $\langle f\rangle$ is not effectively discontinuous if 
the following two requirements are satisfied:
\begin{enumerate}
\item $(\forall \xi<2^{\aleph_0})(\exists q\in \IN^\IN)\left(\Phi_q\Phi_{r_\xi}(q)\in\langle f\rangle(\Phi_{r_\xi}(q))\mbox{ and }\Phi_q\mbox{ total}\right)$,
\item $(\forall \xi<2^{\aleph_0})(\exists p\in\IN^\IN)\;\Phi_{r_\xi}(p)\not\in f(p)$.
\end{enumerate}
The second condition guarantees that $f$ has no continuous realizer $\Phi_{r_\xi}$, the
first condition guarantees that no $\Phi_{r_\xi}$ is a discontinuity function for $\langle f\rangle$.
We build $f$ by transfinite recursion. 
For this purpose we construct two increasing sequences $(N_\xi)_{\xi<2^{\aleph_0}}$ and $(P_\xi)_{\xi<2^{\aleph_0}}$ 
of sets $N_\xi,P_\xi\In\IN^\IN\times\IN^\IN$. The ``negative list'' $N_\xi$ ensures that $f$ is discontinuous and
the ``positive list'' $P_\xi$ ensures that $\langle f\rangle$ is not effectively discontinuous.
The construction will be such that $|N_\xi|=|\xi|$ and $|P_\xi|\leq|\xi|\cdot\aleph_0$.
During the construction we frequently use the axiom of choice without further mention.
The transfinite recursion goes as follows.
We start with $N_0:=P_0:=\emptyset$.
 For each $0<\xi<2^{\aleph_0}$ we first choose
some $p\in\IN^\IN$ that does not appear as a first component in $M:=\bigcup_{\lambda<\xi}(P_\lambda\cup N_\lambda)$
and we define $N_\xi:=\{(p,\Phi_{r_\xi}(p))\}\cup\bigcup_{\lambda<\xi}N_\lambda$.
Such a choice is possible, as $|M|\leq |\xi|+|\xi|\cdot\aleph_0=\max(|\xi|,\aleph_0)<2^{\aleph_0}$ using
the usual rules of cardinal arithmetic \cite[Corollary~3.7.8]{Dev93}.
Secondly, we also choose some $q\in\IN^\IN$ such that $\Phi_q$ is total and such that no $\pi_i\Phi_q\Phi_{r_\xi}(q)$ with $i\in\IN$ 
appears in any second component of $N_\xi$.
Here $\pi_i:\IN^\IN\to\IN^\IN,\langle p_0,p_1,...\rangle\mapsto p_i$ denotes the projection
on the $i$--th component.
Such a choice is possible, as $|N_\xi|\leq|\xi|<2^{\aleph_0}$ and $\Phi_q$ for $q\in\IN^\IN$
includes all the constant total functions.
We define
$P_\xi:=\{(\pi_i\Phi_{r_\xi}(q),\pi_i\Phi_q\Phi_{r_\xi}(q)):i\in\IN\}\cup\bigcup_{\lambda<\xi}P_\lambda$.
 The construction guarantees that $|N_\xi|=|\xi|$ and $|P_\xi|\leq|\xi|\cdot\aleph_0$.
This ends the transfinite recursion.
We now define $P:=\bigcup_{\xi<2^{\aleph_0}}P_\xi$ and $N:=\bigcup_{\xi<2^{\aleph_0}}N_\xi$.
The construction guarantees that $P$ is the graph of a partial problem $f:\In\IN^\IN\mto\IN^\IN$
and $N$ is the graph of a single-valued function $g:\In\IN^\IN\to\IN^\IN$ with $N\cap P=\emptyset$.
Actually, $P$ guarantees that $f$ satisfies condition (1) and $N$ guarantees that $f$ satisfies condition (2).
We still need to extend $P$ to the graph of a total problem $f:\IN^\IN\mto\IN^\IN$ without affecting the conditions (1) and (2). 
For this purpose, we choose for every
$p\in\IN^\IN$ that does not yet appear in a first component of $P$ some $s\in\IN^\IN$ such that
$(p,s)\not\in N$ and we add $(p,s)$ to $P$. This is possible since $N$ is the graph of a single-valued function $g$.
Now $P$ is the graph of a suitable total problem $f:\IN^\IN\mto\IN^\IN$.
\end{proof}

In Corollary~\ref{cor:determinacy} we will see that without the axiom of choice ($\AC$) we cannot construct
discontinuous problems that are not effectively discontinuous.
We close with mentioning that the example in Theorem~\ref{thm:AC} cannot be strengthened to
a single-valued function.

\begin{proposition}[Discontinuous functions]
\label{prop:discontinuous}
Every single-valued ${f:\In\IN^\IN\to\IN^\IN}$ is either continuous or effectively discontinuous.
\end{proposition}
\begin{proof}
By $\LPO:\IN^\IN\to\{0,1\}$ we denote the characteristic function of $\{000...\}$. 
It is well-known that every discontinuous $f:\In\IN^\IN\to\IN^\IN$ satisfies $\LPO\leq_\mathrm{W}^*f$~\cite[Lemma~8.2.6]{Wei00}. It is easy to see that $\DIS\leqSW\LPO$ holds (given an input $p\in\IN^\IN$ for the universal function $\U$,
use $\LPO$ to determine whether $\U(p)=000...$ or not).
The statement now follows with Theorem~\ref{thm:effective-discontinuity}.
\end{proof}

This result can be extended to functions $f:X\to Y$ on admissibly represented spaces $X,Y$ (which can be proved
similarly as \cite[Theorem~4.13]{Pau10a}).

\section{A Game Characterization}
\label{sec:games}

In this section we want to characterize continuity and effective discontinuity using games. 
It is common in descriptive set theory and computability theory to use games to characterize reducibilities and other properties \cite{Kec95,Mos09,Soa16} (see \cite{Tel87,Myc92} for historical surveys).

Wadge~\cite{Wad83} introduced games on subsets $A,B\In\IN^\IN$ to characterize the reducibility that is named after him.
Nobrega and Pauly~\cite{Nob18,NP19} have used a modification of Wadge games for problems $f:\In X\mto Y$ in order to characterize lower cones in the Weihrauch lattice.
We consider similar generalized\footnote{We warn the reader that the extension of the notion of a Wadge game from sets $A,B\In\IN^\IN$
to problems does not automatically mean that determinacy properties carry over.}  versions of Wadge and Lipschitz games, defined as follows. We recall that by ${f^\r:\In\IN^\IN\to\IN^\IN}$ we denote the realizer version
of a problem $f:\In X\mto Y$ (see section~\ref{sec:effective-discontinuity}).

\begin{definition}[Wadge game of problems]
Let $f:\In\IN^\IN\mto\IN^\IN$ be a problem. In a {\em Wadge game} $f$
two players I and II consecutively play words, with player I starting:
\begin{itemize}
\item Player I: $x_0\quad  x_1\quad x_2 \quad...\quad=:x$,
\item Player II: $\quad y_0\quad  y_1\quad y_2 \quad...\quad=:y$,
\end{itemize}
with $x_i,y_i\in\IN^*$. 
The concatenated sequences $(x,y)\in(\IN^\IN\cup\IN^*)^2$ are called a {\em run} of the game $f$.
We say that Player II {\em wins} the run $(x,y)$ of $f$, if
$(x,y)\in\graph(f)$ or $x\not\in\dom(f)$. Otherwise Player I {\em wins}. 
A Wadge game is called a {\em Lipschitz game} if $x_i,y_i\in\IN$.
A {\em Wadge} or {\em Lipschitz game} of a general problem $f:\In X\mto Y$ is understood to be the corresponding game of 
the realizer version $f^\r:\In\IN^\IN\mto\IN^\IN$.
\end{definition}

We allow both players to play arbitrary words, including the empty word.
One can see that player I does not take any advantage of playing words and he could be restricted to natural numbers in a Wadge game, without loss of generality.
Likewise, player II does not take any advantage of playing arbitrary words, it would suffice to allow numbers and the empty word, where the empty word essentially corresponds to skipping the corresponding move.
This shows that our notion of a Wadge game for problems corresponds to the one of Nobrega and Pauly~\cite{Nob18,NP19}.
For simplicity we have allowed arbitrary words for both players.
This does not only lead to a more symmetric definition, but it also simplifies the proof of Theorem~\ref{thm:Wadge-game} below.
As usual we define winning strategies for games to be word functions that determine moves for one player depending on the
moves of the other player. 

\begin{definition}[Winning strategy]
Let $f:\In X\mto Y$ be a problem and let ${\sigma:\IN^{**}\to\IN^*}$ be a function.
We consider the Wadge game $f$.
\begin{enumerate}
\item
$\sigma$ is called a {\em winning strategy} for Player II in the game $f$, if Player II 
wins every run of $f$ with her moves being
determined by 
\[y_i:=\sigma(x_0,...,x_i),\] 
while Player I plays $x_0,x_1,...\in\IN^*$. 
\item
$\sigma$ is called a {\em winning strategy} for Player I in the game $f$, if Player I
wins every run of $f$
with his moves being determined by 
\[x_i:=\sigma(y_0,...,y_{i-1}),\] 
while Player II plays $y_0,y_1,...\in\IN^*$. 
\end{enumerate}
Winning strategies for Lipschitz games are defined analogously with functions of type $\sigma:\IN^*\to\IN$.
\end{definition}

Now our main observation on Wadge games is that winning of player II characterizes continuity of the problem
and winning of player I effective discontinuity. Nobrega and Pauly have proved a general version of the first observation for lower cones in the Weihrauch lattice~\cite[Theorem~3.3]{NP19}. 
We use some bijective standard numbering $\w:\IN\to\IN^*$ and we use the notation $\overline{v}:=\w^{-1}(v)$ for all $v\in\IN^*$.

\begin{theorem}[Wadge games]
\label{thm:Wadge-game}
We consider the Wadge game of a given problem ${f:\In X\mto Y}$.
Then the following hold:
\begin{enumerate}
\item $f$ is continuous $\iff$ Player II has a winning strategy for $f$,
\item $f$ is effectively discontinuous $\iff$ Player I has a winning strategy for $f$.
\end{enumerate}
\end{theorem}
\begin{proof}
Since the Wadge game of $f$ is the Wadge game of $f^\r$ and $f^\r\equivSW f$, it suffices by 
Theorem~\ref{thm:dichotomy} to consider problems of type $f:\In\IN^\IN\mto\IN^\IN$.\smallskip\\
(1) If $f$ is continuous, then $f$ has a continuous realizer $F:\In\IN^\IN\to\IN^\IN$, which is approximated
by a monotone function $h:\IN^*\to\IN^*$ in the sense that $F(p)=\sup_{w\prefix p}h(w)$ for all $p\in\dom(F)\supseteq\dom(f)$. Given the moves $w_0,w_1,...\in\IN^*$ of Player~I, we can
inductively define the moves $v_i\in\IN^*$ by $v_0...v_i:=h(w_0...w_i)$ for all $i\in\IN$, since $h$ is monotone.
Then $\sigma(\varepsilon):=\varepsilon$ and $\sigma(w_0,...,w_i):=v_i$ provides a winning strategy
$\sigma:\IN^{**}\to\IN^*$ for Player~II.
This is because if $r:=w_0w_1...\in\dom(f)$, then $F(r)=v_0v_1...\in f(r)$.\smallskip\\
Vice versa, let $\sigma:\IN^{**}\to\IN^*$ be a winning strategy for Player~II.
Then we can define $h:\IN^*\to\IN^*$ by $h(\varepsilon):=\varepsilon$ and $h(a_0...a_i):=v_0...v_i$ for all $a_0,a_1,...\in\IN$,
where we inductively choose  $v_i:=\sigma(a_0,...,a_i)$.
Let $F:\In\IN^\IN\to\IN^\IN$ be given by $F(p):=\sup_{w\prefix p}h(w)$.
Given an input $r:=a_0a_1...\in\dom(f)$, we obtain $F(r)=v_0v_1...$ such that
$(r,F(r))\in\graph(f)$, since $\sigma$ is a winning strategy for Player II. 
Hence $F$ is a continuous realizer for $f$.\medskip\\
(2) If $f$ is effectively discontinuous, then there is a continuous $D:\IN^\IN\to\IN^\IN$
that witnesses the discontinuity of $f$ in the sense that $D(p)\in\dom(f)$ and
$\Phi_pD(p)\not\in fD(p)$ for all $p\in\IN^\IN$.
Let $h:\IN^*\to\IN^*$ be a monotone function that approximates $D$, in the sense
that $D(p)=\sup_{w\prefix p}h(w)$ for all $p\in\IN^\IN$.
Given the moves $v_0,v_1,...\in\IN^*$ of Player II, we can inductively
define the moves $w_i$ by $w_0:=h(\varepsilon)$ and  $w_0...w_i:=h(\langle\overline{w_0},\overline{v_0}\rangle...\langle\overline{w_0...w_{i-1}},\overline{v_0...v_{i-1}}\rangle)$. 
We note that $r:=w_0w_1...$ is an infinite sequence $r\in\dom(f)$,
since $D$ is total and $\range(D)\In\dom(f)$. And $p:=\langle\overline{w_0},\overline{v_0}\rangle\langle\overline{w_0w_{1}},\overline{v_0v_{1}}\rangle...\in\IN^\IN$
is the name of a function $\Phi_p$ with $D(p)=r$. 
If $q:=v_0v_1...$ is finite, then clearly $(r,q)\not\in\graph(f)$. Otherwise, $q=\Phi_p(r)=\Phi_pD(p)\not\in fD(p)=f(r)$
and hence also $(r,q)\not\in\graph(f)$. This means that $\sigma:\IN^{**}\to\IN^*$
with $\sigma(v_0,...,v_{i-1}):=w_i$ is a winning strategy for Player~I.\smallskip\\
Vice versa let  $\sigma:\IN^{**}\to\IN^*$ be a winning strategy for Player~I.
We need to define a continuous discontinuity function $D:\IN^\IN\to\IN^\IN$ for $f$.
Given $p\in\IN^\IN$, we can determine a monotone $h:\IN^*\to\IN^*$ that approximates $\Phi_p$, i.e., 
such that $\Phi_p(q)=\sup_{w\prefix q}h(w)$ for all $q\in\dom(\Phi_p)$. 
We let $D(p):=w_0w_1...$
where the $w_i$ are inductively given by $w_i:=\sigma(v_0,...,v_{i-1})$
and $v_0...v_i:=h(w_0...w_i)$.
Since $\sigma$ is a winning strategy for Player~I, we have $D(p)\in\dom(f)$. In particular, $D$ is total and continuous.
Moreover, with $q:=v_0v_1...$ we have $(D(p),q)\not\in\graph(f)$.
This could mean that $q$ is finite and hence $D(p)\not\in\dom(\Phi_p)$ 
or otherwise $\Phi_pD(p)=q$. In any case, $\Phi_pD(p)\not\in fD(p)$ holds and $D$ is a discontinuity function
for $f$.
\end{proof}

In passing, we note that the proof is fully constructive in the sense that computable winning strategies
translate into computable functions and vice versa in the following sense.

\begin{corollary}[Wadge games]
\label{cor:Wadge-game}
We consider the Wadge game of a given problem ${f:\In X\mto Y}$.
Then the following hold:
\begin{enumerate}
\item $f$ is computable $\iff$ Player~II has a computable winning strategy for $f$,
\item $f$ is computably discontinuous $\iff$ Player~I has a computable winning strategy for $f$.
\end{enumerate}
\end{corollary}

One reason why it is useful to have characterizations of continuity and effective discontinuity
in game form is that for certain games determinacy conditions are known and well understood.
The axiom of determinacy ($\AD$), which was introduced by Mycielski and Steinhaus~\cite{MS62,Myc64} 
states that every Gale-Stewart game is determined.
We note that this axiom is inconsistent with the axiom of choice.
We recall the definition of Gale-Stewart games~\cite{Soa16}.\footnote{Such games over binary digits were already considered by Ulam~\cite{Tel87}.}

\begin{definition}[Gale-Stewart game]
Let $A\In\IN^\IN$. Then in a {\em Gale-Stewart game} $A$
two players I and II consecutively play numbers
\begin{itemize}
\item Player~I: $x_0\quad  x_1\quad x_2 \quad...\quad=:x$,
\item Player~II: $\quad y_0\quad  y_1\quad y_2 \quad...\quad=:y$,
\end{itemize}
with $x_i,y_i\in\IN$. 
The concatenated sequence $r=\langle x,y\rangle\in\IN^\IN$ is called a {\em run} of the game $A$.
We say that Player~II {\em wins} the run $r$ of $A$, if
$r\in A$. Otherwise Player~I {\em wins}. 
\end{definition}

What we have described as the Gale-Stewart game of $A$ is usually considered as the Gale-Stewart game of the complement $\IN^\IN\setminus A$.
We prefer the complementary version as it fits better to our definition of
Wadge games.

Winning strategies for Gale-Stewart games can be defined analogously to Lipschitz games. 
In fact, Lipschitz games of problems $f$ are essentially Gale-Stewart games on $\graph(f)$,
at least for total problems on Baire space. For general problems, the difference can be expressed using
totalizations of problems, a concept that was studied in \cite{NP18,BG21a}.

\begin{definition}[Totalization]
Let $f:\In X\mto Y$ be a problem. Then the {\em totalization} $\T f:X\mto Y$ is defined by
\[\T f(x):=\left\{\begin{array}{ll}
f(x) & \mbox{if $x\in\dom(f)$}\\
Y & \mbox{otherwise}
\end{array}\right.\]
\end{definition}

Now we can express the relation between Lipschitz games and Gale-Stewart games on graphs
as follows.

\begin{proposition}[Lipschitz games and Gale-Stewart games]
\label{prop:Lipschitz-Gale-Stewart}
Let $f:\In \IN^\IN\mto\IN^\IN$ be a problem. Then the winning strategies for either player in the Lipschitz game $f$ are identical to the winning strategies 
of the corresponding player for the Gale-Stewart game $\langle\graph(\T f)\rangle$.
\end{proposition}
\begin{proof}
The proof follows from the easy observation that 
\[\graph(\T f)=\graph(f)\cup\left((\IN^\IN\setminus\dom(f))\times\IN^\IN\right)\]
and this is exactly the payoff set for player II in the Lipschitz game $f$.
\end{proof}

We note that $f\leqW\T f$, but in general $\T f$ is not Weihrauch equivalent to $f$~\cite{BG21a}.

For general Wadge games the situation is somewhat more subtle as players can play empty words.
We can, however, bridge the step between Wadge and Lipschitz games by coding words in numbers.
We consider a canonical bijective standard numbering $\w:\IN\to\IN^*,i\mapsto\w_i$ of words over natural numbers, 
which we lift to a partial function
\[\w:\In\IN^\IN\to\IN^\IN,p\mapsto\w_{p(0)}\w_{p(1)}\w_{p(2)}...\]
that we also denote by $\w$. Here $\dom(\w)=\{p\in\IN^\IN:\w_{p(0)}\w_{p(1)}\w_{p(2)}...$ is infinite$\}$
and hence $\w$ is not total (since there is a number $n$ that encodes the empty word).
Now we can consider the problem $f^\w:=\w^{-1}\circ f\circ\w$, which is $f$ lifted to numbers (that encode words).
We note that $f^\w$ is not total, even if $f$ is total, since $\w$ is not total.
We mention that $\w:\In\IN^\IN\to\IN^\IN$ can be seen as a precomplete representation of $\IN^\IN$ (in fact, as an alternative way to define the {\em precompletion} of $\id:\IN^\IN\to\IN^\IN$, as studied in \cite{BG20,BG21a}). 
Since $\w$ is computably equivalent to $\id$ as a representation, we
obtain $f\equivSW f^\w$.
Now we can express our observation as follows.

\begin{proposition}[Wadge games and Lipschitz games]
\label{prop:Wadge-Lipschitz}
Let $f:\In\IN^\IN\mto\IN^\IN$ be a problem and let $P\in\{\mathrm{I, II}\}$. Then the following are equivalent.
\begin{enumerate}
\item Player~$P$ has a winning strategy for the Wadge game $f$.
\item Player~$P$ has a winning strategy for the Lipschitz game $f^\w$.
\end{enumerate}
\end{proposition}
\begin{proof}
Given a winning strategy $\sigma:\IN^{**}\to\IN^*$ for player $P$ of the Wadge game $f$, we have to convert this strategy
into a strategy $\lambda:\IN^*\to\IN$ for player $P$ of the Lipschitz game $f^\w$.
For this purpose we just have to define $\lambda$ such that $\lambda(n_0,...,n_k):=\w^{-1}\sigma(\w_{n_0},...,\w_{n_k})$.
Vice versa, if $\lambda:\IN^*\to\IN$ is a winning strategy for player $P$ of the Lipschitz game $f^\w$, then
$\sigma(w_0,...,w_k):=\w_{\lambda(\w^{-1}(w_0),...,\w^{-1}(w_k))}$ defines a winning strategy $\sigma:\IN^{**}\to\IN^*$
for player $P$ of the Wadge game $f$.
\end{proof}

If we combine Propositions~\ref{prop:Lipschitz-Gale-Stewart} and \ref{prop:Wadge-Lipschitz} then we obtain the following
result.

\begin{corollary}[Wadge games and Gale-Stewart games]
\label{cor:Wadge-Gale-Stewart}
Let $f:\In X\mto Y$ be a problem and let $P\in\{\mathrm{I, II}\}$. Then the following are equivalent.
\begin{enumerate}
\item Player~$P$ has a winning strategy for the Wadge game $f$,
\item Player~$P$ has a winning strategy for the Lipschitz game $f^\mathrm{rw}$,
\item Player~$P$ has a winning strategy for the Gale-Stewart game $\langle\graph(\T(f^\mathrm{rw}))\rangle$.
\end{enumerate}
\end{corollary}

We note that $\T(f^\mathrm{rw})\equivSW\overline{f}$, where $\overline{f}$ denotes the so-called {\em completion} of $f$, see \cite[Lemma~4.9]{BG21a}.
Hence, the operation of precompletion moves us from Wadge games to Lipschitz games and the operation of completion to Gale-Stewart games. 

In order to have some simple terminology at hand, we introduce the following notions of determinacy.

\begin{definition}[Determinacy]
We call a problem $f:\In X\mto Y$ {\em determined}, if it is either continuous or effectively discontinuous.
By {\em Weihrauch determinacy} we understand the property that every problem $f$ is determined.
\end{definition}

Hence, by Theorem~\ref{thm:effective-discontinuity} $f$ is determined if either $f\leq_\mathrm{W}^*\id$ or $\DIS\leq_\mathrm{W}^* f$ holds and
by Theorem~\ref{thm:Wadge-game} this is the case if and only if the Wadge game $f$ is determined in the sense that either player I or player II has a winning strategy.
Weihrauch determinacy means that every $f$ satisfies the above dichotomies. 

Since the axiom of determinacy $\AD$ states that every Gale-Stewart game is determined, i.e., either player I or player II has a winning strategy,
we immediately get the following conclusion of Corollary~\ref{cor:Wadge-Gale-Stewart}
with the help of Theorem~\ref{thm:effective-discontinuity} (see also \cite[Corollary~3.7]{NP19}.

\begin{corollary}[Determinacy]
\label{cor:determinacy}
$\ZF+\DC+\AD$ implies Weihrauch determinacy.
\end{corollary}

We can conclude more from Corollary~\ref{cor:Wadge-Gale-Stewart}. Namely, in $\ZF+\AC$ it is known by a Theorem of Martin~\cite[Theorem~20.5]{Kec95} that every
Borel set $A\In\IN^\IN$ is determined, i.e., either player I or player II has a winning strategy
(see also \cite[Corollary~3.6]{NP19}).

\begin{proposition}[Borel determinacy]
\label{prop:Borel-determinacy}
In $\ZFC$ every problem $f:\In X\mto Y$ such that $\graph(f^\r)\In\IN^\IN\times\IN^\IN$ and $\dom(f^\r)\In\IN^\IN$ are Borel sets, is determined.
\end{proposition}
\begin{proof}
Firstly, for a problem $f:\In\IN^\IN\mto\IN^\IN$ we note that 
\[\graph(f^\w)=\graph(\w^{-1}\circ f\circ \w)=(\w\times\w)^{-1}(\graph(f))\] 
and $\dom(f^\w)=\dom(f\circ\w)=\w^{-1}(\dom(f))$.
Since $\w$ is continuous with a $\PO{2}$--domain $\dom(\w)$, it follows that $\graph(f^\w)$ and $\dom(f^\w)$ are Borel measurable, 
if $\graph(f)$ and $\dom(f)$ are so. Hence, $\langle\graph(\T f^\w)\rangle$ is Borel measurable in this situation.
For a general problem $f:\In X\mto Y$, we can apply the previous considerations to $f^\r:\In\IN^\IN\mto\IN^\IN$ and we obtain the claim
using Borel determinacy and Corollary~\ref{cor:Wadge-Gale-Stewart}.
\end{proof}

That is even under the axiom of choice $\AC$ examples of problems that are discontinuous but not effectively so have to be rather complicated.
As Polish spaces admit continuous and total versions of the Cauchy representation (see, e.g., \cite[Corollary~4.4.12]{Bra98b}),
we also obtain the following version of the previous corollary.

\begin{corollary}[Borel determinacy on Polish spaces]
\label{cor:Borel-determinacy-polish}
Let $X,Y$ be Polish spaces.
In $\ZFC$ every problem ${f:\In X\mto Y}$ such that $\graph(f)\In X\times Y$ and $\dom(f)\In X$ are Borel sets, is determined.
\end{corollary}
\begin{proof}
For Polish spaces $X$ and $Y$ we can assume that the Cauchy representations 
$\delta_X:\IN^\IN\to X$ and $\delta_Y:\IN^\IN\to Y$ are total~\cite[Corollary~4.4.12]{Bra98b}.
They are also continuous maps.
Since $f^\r=\delta_Y^{-1}\circ f\circ\delta_X$, this implies that 
$\graph(f^\r)=(\delta_X\times\delta_Y)^{-1}(\graph(f))$ and
$\dom(f^\r)=\delta_X^{-1}(\dom(f))$ are Borel sets, if $\graph(f)$ and $\dom(f)$ are Borel sets.
Now the claim follows with Proposition~\ref{prop:Borel-determinacy}.
\end{proof}

It is an interesting questions whether we get the inverse implication in Corollary~\ref{cor:determinacy} in the following sense.

\begin{question}
\label{question:AD}
Does $\AD$ follow from Weihrauch determinacy in $\ZF+\DC$?
\end{question}

We note that a similar question, namely whether Wadge determinacy (in the usual sense of Wadge games for sets $A,B\In\IN^\IN$) implies $\AD$ is non-trivial and the answer is not known~\cite{And03,And06}.
For our generalized Wadge games the situation seems to be simpler and we can easily obtain $\AD$ from
Lipschitz determinacy of problems (with a non-constructive proof).

\begin{proposition}
\label{prop:Lipschitz-AD}
In $\ZF+\DC$ the following are equivalent:
\begin{enumerate}
\item The axiom of determinacy $\AD$.
\item Every Lipschitz game $f:\In X\mto Y$ is determined,
i.e., either player I or player II has a winning strategy.
\end{enumerate}
\end{proposition}
\begin{proof}
That (1) implies (2) was proved in Proposition~\ref{prop:Lipschitz-Gale-Stewart}. We still need to prove that (2) implies (1).
For every set $G\In\IN^\IN$ we define the problem $f_G:\In\IN^\IN\mto\IN^\IN$ with $f_G(p):=\{q\in\IN^\IN:\langle p,q\rangle\in G\}$
and $\dom(f_G):=\{p\in\IN^\IN:(\exists q\in \IN^\IN)\;\langle p,q\rangle\in G\}$. We call $G$ {\em total}, if $\dom(f_G)=\IN^\IN$.
It is clear that for non-total $G$ player I in the Gale-Stewart game $G$ always has a winning strategy, he just needs to play some
$p\not\in\dom(f_G)$. For total $G$ we obtain $\langle\graph(\T f_G)\rangle=\langle \graph(f_G)\rangle=G$. Hence, by Proposition~\ref{prop:Lipschitz-Gale-Stewart}
$G$ is determined if the Lipschitz game $f_G$ is so. This proves determinacy of $G$ for every $G\In\IN^\IN$.
\end{proof}

In order to answer Question~\ref{question:AD} positively one still needs to bridge the gap between Lipschitz determinacy and Wadge determinacy for problems $f$
(possibly using the techniques from~\cite{And03,And06} or a simpler argument).

We prove that Weihrauch determinacy implies at least Wadge determinacy, in the original sense of Wadge games for sets $A,B\In\IN^\IN$.
We can simulate Wadge games for sets $A,B$ using a particular problem $\frac{B}{A}$ (which was introduced in \cite{NP19}).

\begin{definition}[Wadge game for sets]
Let $A,B\In\IN^\IN$. Then the Wadge game $(A,B)$ is the Wadge game of the problem $\frac{B}{A}:\IN^\IN\mto\IN^\IN$ with
\[\graph\left(\frac{B}{A}\right):=(A\times B)\cup\left((\IN^\IN\setminus A)\times(\IN^\IN\setminus B)\right).\]
\end{definition}

In the Wadge game $\frac{B}{A}$ player II wins if the corresponding run $(x,y)$ satisfies $x\in A\iff y\in B$ and otherwise player I wins.
This shows that the Wadge game $(A,B)$ is the usual one and as usually, by {\em Wadge determinacy} we understand the property that
every Wadge game $(A,B)$ is determined, i.e., either player I or player II has a winning strategy.
This immediately yields the following.

\begin{corollary}[Wadge determinacy]
\label{cor:Wadge}
In $\ZF+\DC$ Weihrauch determinacy implies Wadge determinacy.
\end{corollary}

Another conclusion that we can either directly draw from Proposition~\ref{prop:embedding-effectively-discontinuous} without further ado 
or from Corollary~\ref{cor:Wadge} with the help of \cite[Theorem~II.C.2]{Wad83}
is the following.

\begin{corollary}[Perfect subset property]
\label{cor:PSP}
In $\ZF+\DC$ Weihrauch determinacy implies that every set $A\In\IN^\IN$ satisfies the perfect subset property.
\end{corollary}

If Question~\ref{question:AD} has a negative answer or cannot be answered easily, then one could ask other questions
such as the following.

\begin{question}
Does the Baire property $\BP$ follow from Weihrauch determinacy in $\ZF+\DC$?
\end{question}

By the Baire property $\BP$ we mean that statement that every subset $A\In\IN^\IN$ can be written
as symmetric difference $A=U\Delta M$ with an open set $U\In\IN^\IN$ and a meager set $M\In\IN^\IN$.
If the answer to this question is also negative, then one can ask the following modified version 
of Question~\ref{question:AD}.

\begin{question}
Does $\AD$ follow from Weihrauch determinacy in $\ZF+\DC+\BP$?
\end{question}

A positive answer to this question is not unlikely in light of Proposition~\ref{prop:Lipschitz-AD}, 
given that Wadge determinacy and Lipschitz determinacy
(both in the usual sense of games for sets $A,B\In\IN^\IN$) have been proven to be equivalent in
$\ZF+\DC+\BP$ by Andretta~\cite{And03,And06}.

\section{Computable Discontinuity and Productivity}
\label{sec:characteristic}

In this section we briefly want to discuss the question what computable discontinuity means for subsets and how
the notion is linked to the notion of productivity.
Here, for every subset $A\In X$ of some fixed space $X$, we denote by
\[\chi_A:X\to\IS,x\mapsto\left\{\begin{array}{ll}
1 & \mbox{if $x\in A$}\\
0 & \mbox{otherwise}
\end{array}\right.\]
the {\em characteristic function} of $A$.
The codomain $\IS=\{0,1\}$ is Sierpi\'nski space that is equipped with the total representation
$\delta_\IS$ with $\delta_\IS(p)=0:\iff p=000...$. 
It is well-known that for any represented space $X$ and $A\In X$ the characteristic function
$\chi_A$ is continuous if and only if $A$ is open.

As a side remark we mention that the characteristic function in this form can also nicely be used to embed Wadge reducibility and many-one reducibility
into the strong version of Weihrauch reducibility.
We recall that for $A,B\In\IN^\IN$ the set $A$ is called {\em Wadge reducible} to $B$, in symbols $A\leqW B$, if there is a continuous
$f:\IN^\IN\to\IN^\IN$ with $A=f^{-1}(B)$. Likewise, we say for $A,B\In\IN$ that $A$ is {\em many-one reducible} to $B$,
in symbols $A\leqm B$, if there is a computable $f:\IN\to\IN$ such that $A=f^{-1}(B)$. Since the only continuous functions $f:\IS\to\IS$ are the identity and the two constant functions, we directly get the following characterization of Wadge and many-one reducibility in terms of (continuous) strong Weihrauch reducibility.

\begin{proposition}[Reducibilities]
We obtain
\begin{enumerate}
\item $\chi_A\leqSW\chi_B\iff A\leqm B$ for all $A,B\In\IN$.
\item $\chi_A\leq_\mathrm{sW}^*\chi_B\iff A\leqW B\iff \frac{B}{A}\leq_\mathrm{W}^*\id$ for all $A,B\In\IN^\IN$.
\end{enumerate}
\end{proposition}

Hence, the corresponding reducibility structures can be embedded into the corresponding strong versions of the Weih\-rauch lattice.
By $\OO(X)$ we denote the set of open subsets of $X$ equipped with the representation $\delta_{\OO(X)}$,
defined by $\delta_{\OO(X)}(p)=U:\iff\Phi_p$ is a realizer of $\chi_U$.
Now one can ask what it means for $\chi_A$ to be computably discontinuous.
The following result answers this question for Baire space $X=\IN^\IN$ using the {\em symmetric difference problem}
$\Delta_A:\In\OO(X)\mto X,U\mapsto A\Delta U$. Here $A\Delta U:=(A\setminus U)\cup(U\setminus A)$ denotes
the {\em symmetric difference} of $A, U\In X$. The problem $\Delta_A$ is total for non-open $A\In X$.

\begin{proposition}[Symmetric difference]
\label{prop:symmetric-difference}
Let $A\In\IN^\IN$ not be open. Then the problem $\Delta_A$ is computable if and only if $\chi_A$ is computably discontinuous.
\end{proposition}
\begin{proof}
If $\chi_A:\IN^\IN\to\IS$ is computably discontinuous, then there is a computable discontinuity function $D:\IN^\IN\to\IN^\IN$
for $\chi_A$. Given a name $q$ of a realizer $\Phi_q$ of the continuous function $\chi_U:\IN^\IN\to\IS$ for any given $U\in\OO(\IN^\IN)$,
the value $D(q)$ is a name for some point $x\in X$ with $\chi_U(x)\not=\chi_A(x)$, i.e., $x\in A\Delta U$. 
That is, $\Delta_A$ is realized by $D$ and hence computable.
On the other hand, let $\Delta_A$ be computable and $A\In \IN^\IN$ not open.
Given a name $q$ of some potential realizer $\Phi_q$ of $\chi_A$, we can convert $q$ computably into a name $r=G(q)$ of some total function
$\Phi_r$ with $\delta_\IS\Phi_r=\delta_\IS\Phi_q$, because $\delta_\IS$ is precomplete (the idea is that $\Phi_r$ produces zero output as long as $\Phi_q$ makes no other information available).
Let us denote by $G:\IN^\IN\to\IN^\IN$ the corresponding computable function.
Now there is some open $U\In \IN^\IN$ such that $\Phi_r$ is a realizer of $\chi_U$, namely $U:=\Phi_r^{-1}(\IN^\IN\setminus\{000...\})=\Phi_r^{-1}\delta_\IS^{-1}(\{1\})$.
Since $A$ is not open, $\Delta_A(U)$ is defined and non-empty.
Let $F$ be a computable realizer of $\Delta_A$ and $D:=FG$. Then $F(r)\in\Delta_A(U)=(A\setminus U)\cup(U\setminus A)$, 
i.e., $\chi_UD(q)\not=\chi_AD(q)$. This implies $\delta_\IS\Phi_qD(q)\not=\chi_AD(q)$, since either $D(q)\not\in\dom(\Phi_q)$ or otherwise
$\delta_\IS\Phi_qD(q)=\delta_\IS\Phi_rD(q)=\chi_UD(q)$. That is $D$ is a computable discontinuity function of $\chi_A$.
\end{proof}

An analogous statement holds if we replace computable by continuous in both occurrences. However,
in this case the statement is void as any $\chi_A$ for non open $A$ is effectively discontinuous
by (a suitable extension of) Proposition~\ref{prop:discontinuous}.
We can also summarize Proposition~\ref{prop:symmetric-difference} as follows.

\begin{corollary}[Symmetric difference]
\label{cor:symmetric-difference}
$\Delta_A\leqW\id\iff\DIS\leqW\chi_A$ for non-open $A\In\IN^\IN$.
\end{corollary}

Proposition~\ref{prop:symmetric-difference} 
shows that the notion of computable discontinuity is also formally related to the notion of productivity. 
Weihrauch has introduced a topological and a computability theoretic notion of productivity for subsets $A\In\IN^\IN$~\cite[Definition~4.4]{Wei85},
\cite[Definition~3.2.24]{Wei87}.
If, in a similar way, we transfer the definition of {\em completely productive sets}, as originally defined by Dekker~\cite{Dek55},
then we could say that non-open $A\In \IN^\IN$ is {\em completely productive} if $\Delta_A$ is computable.
Using this terminology, non-open $A\In \IN^\IN$ is completely productive if and only if $\chi_A$ is computably discontinuous by Proposition~\ref{prop:symmetric-difference}.
Myhill~\cite{Myh55} proved that a set $A\In\IN$ is productive if and only if $A$ is completely productive (see also the proofs in~\cite[Theorem~VII, \S 11.3]{Rog67},  \cite[Corollary~2.6.8]{Wei87}).
It is clear that complete productivity for $A\In\IN^\IN$ implies productivity.
We leave it as a task to the reader to study whether productivity also implies complete productivity for Baire space $\IN^\IN$ (or even more general spaces).

\section{Conclusions}
\label{sec:conclusions}

We have introduced the discontinuity problem $\DIS$ and we have provided some evidence that one can consider it
as the simplest natural unsolvable problem with respect to the continuous version of Weihrauch reducibility. 
At least in $\ZF+\DC+\AD$ it turns out that it actually induces the minimal discontinuous Weihrauch degree.

More results on the discontinuity problem will be provided in a forthcoming article~\cite{Bra20b}.
While the original definition of $\DIS$ is in terms of a universal function $\U$,
it is useful to have a characterization in purely set-theoretic terms. 
In \cite{Bra20b} we prove that the discontinuity problem is equivalent to the {\em range non-equality problem}
defined by 
$\NRNG:\IN^\IN\mto2^\IN,p\mapsto\{A\in2^\IN:A\not=\range(p-1)\}$.
Here for $p\in\IN^\IN$ the finite or infinite sequence $p-1\in\IN^\IN\cup\IN^*$ 
is the sequence that is obtained as concatenation of $p(0)-1, p(1)-1, p(2)-1,...$,
where $-1$ is identified with the empty word $\varepsilon$.

\begin{proposition}
$\DIS\equivSW\NRNG$.
\end{proposition}

In \cite{Bra20b} we also discuss algebraic properties of the discontinuity problem $\DIS$ and one important property is that
the parallelization $\widehat{\DIS}$ of the discontinuity problem is equivalent to the non-computability problem $\NON$.

\begin{theorem}[Non-computability is parallelized discontinuity]
$\NON\equivSW\widehat{\DIS}$.
\end{theorem}

The {\em non-computability problem} is defined with the help of Turing reducibility $\leqT$ by
$\NON:\IN^\IN\mto\IN^\IN,p\mapsto\{q\in\IN^\IN:q\not\leq_\mathrm{T} p\}$.
This result supports the slogan that ``non-computability is the parallelization of discontinuity'' and underlines
that the discontinuity problem is a natural one. 
The discontinuity problem can itself be obtained by summation (a dual operation to parallelization introduced in~\cite{Bra20b})
starting from other natural problems such as $\LPO,\LLPO$ etc. Hence, it is nicely related in an algebraic
way to other natural problems in the Weihrauch lattice.

\bibliographystyle{plain}
\bibliography{C:/Users/\input{"|kpsewhich -var-value USERNAME"}/Dropbox/Bibliography/lit}

\section*{Acknowledgments}

We would like to thank Matthias Schr\"oder for a discussion on the relevance of the axiom of determinacy and different forms of the axiom of choice for computable analysis.
Likewise, we would like to thank Arno Pauly for a discussion on Wadge games for problems
that has helped to improve the corresponding results.
This work has been supported by the {\em National Research Foundation of South Africa} (Grant Number 115269).

\end{document}